\documentclass{amsart}
\usepackage[utf8]{inputenc}
\usepackage[margin=1in]{geometry}
\usepackage{upgreek,amsthm,bbm,bm}
\usepackage{color,amssymb, comment, mathrsfs,tikz,amsmath,amsfonts,stmaryrd,tikz-cd,hyperref,appendix,cancel,faktor,verbatim}
\usepackage{amsmath,amsthm,amssymb}
\usepackage{tikz,tikz-cd,color}
\usepackage{enumitem}
\usepackage{mathtools}
\usepackage{comment}
\usetikzlibrary{matrix}
\usetikzlibrary{arrows}
\usepackage[all]{xy}
\usepackage{amsmath, nccmath,colonequals}
\usepackage{geometry}

\hypersetup{
    colorlinks=true,
    linkcolor=blue,
    filecolor=magenta,      
    urlcolor=cyan,
} 

\usepackage[capitalise]{cleveref}
\crefformat{equation}{(#2#1#3)}
\crefrangeformat{equation}{(#3#1#4--#5#2#6)}
\crefformat{enumi}{(#2#1#3)}
\crefrangeformat{enumi}{(#3#1#4--#5#2#6)}

\newtheorem{theorem}{Theorem}[section]
\newtheorem{lemma}[theorem]{Lemma}
\newtheorem{proposition}[theorem]{Proposition}

\theoremstyle{definition}

\theoremstyle{remark}
\newtheorem{remark}[theorem]{Remark}

\newcommand{\kk}{\Bbbk}

\newcommand{\sqf}{\mathrm{sqf}}
\renewcommand{\b}{\bullet}

\title{A DG $\Gamma$-structure on the Generalized Taylor Resolution}

\usepackage[capitalise]{cleveref}
\crefformat{equation}{(#2#1#3)}
\crefrangeformat{equation}{(#3#1#4--#5#2#6)}
\crefformat{enumi}{(#2#1#3)}
\crefrangeformat{enumi}{(#3#1#4--#5#2#6)}

\author[L.~Ferraro]{Luigi Ferraro}
\address{School of Mathematical and Statistical Sciences,
University of Texas Rio Grande Valley, Edinburg, TX 78539, U.S.A.}
\email{luigi.ferraro@utrgv.edu}

\author[R.~ Alvarez]{Raul Alvarez}
\address{School of Mathematical and Statistical Sciences,
University of Texas Rio Grande Valley, Edinburg, TX 78539, U.S.A.}
\email{raul.alvarez03@utrgv.edu}

\thanks{Luigi Ferraro was partly supported by the Simons Foundation grant MPS-TSM-00007849.}

\keywords{polynomial rings, monomial ideals, free resolutions, Taylor resolution, generalized Taylor resolution, divided powers, differential graded algebras, squarefree monomial ideals}

\subjclass[2020]{13D02}

\begin{document}

\begin{abstract}
Free resolutions of ideals in commutative rings provide valuable insights into the complexity of these ideals. In 1966, Taylor constructed a free resolution for monomial ideals in polynomial rings, which Gemeda later showed admits a differential graded (DG) algebra structure. In 2002, Avramov proved that the Taylor resolution admits a DG algebra structure with divided powers.
In 2007, Herzog introduced the generalized Taylor resolution, which is usually smaller than the original Taylor resolution. Recently, in 2023, VandeBogert showed that the generalized Taylor resolution also admits a DG algebra structure. In this paper, we extend Avramov's result by proving that the generalized Taylor resolution admits a DG algebra structure with divided powers. Along the way, we correct a sign error in VandeBogert's product formula. We provide applications concerning homotopy Lie algebras and resolutions of squarefree monomial ideals.
\end{abstract}

\maketitle

\section{Introduction}
Let $R = \kk[x_1, \dots, x_n]$ be a polynomial ring over a field $\kk$, and let $I \subseteq R$ be a monomial ideal. The construction of free resolutions of $R/I$ is a central theme in combinatorial commutative algebra. Among the most well-studied examples is the Taylor resolution \cite{Taylor}, which provides a canonical $\mathbb{Z}^n$-graded free resolution for any monomial ideal. Although rarely minimal, the Taylor resolution is combinatorially explicit and it plays a foundational role in the study of resolutions of monomial ideals.

A natural question in this context is whether one can enhance such resolutions with additional algebraic structure. Differential graded (DG) algebra structures on resolutions have proved particularly fruitful. They were initially studied by Buchsbaum and Eisenbud in \cite{BE}, where they prove that a finite free resolution of a cyclic module always admits a multiplicative structure associative only up to homotopy. In general associative products, i.e. DG algebra structures, on free resolutions of cyclic modules do not exist, see Avramov \cite{obstructions}.

Beyond DG algebras, it is often useful to consider DG algebras equipped with \emph{divided powers}, also known as DG $\Gamma$-algebras. Free resolutions of cyclic modules with a DG $\Gamma$-algebra structure arise naturally in the study of homotopy Lie algebras.

In \cite{Gemeda}, Gemeda showed that the Taylor resolution admits a DG algebra structure via a simple combinatorial product. Later, Avramov \cite{Avramov} established that the Taylor resolution also admits a canonical structure of a DG $\Gamma$-algebra in which the divided powers of basis elements vanish beyond the first power. Avramov used this structure on the Taylor resolution to study the homotopy Lie algebra of rings of the form $R/I$ where $I$ is a monomial ideal.

A generalization of the Taylor complex was introduced by Herzog \cite{herzog2007} as a construction for resolving cyclic modules of the form $R/(I_1 + \cdots + I_r)$, where $I_1,\ldots,I_r$ are monomial ideals, from given resolutions of the cyclic modules $R/I_i$. This \emph{generalized Taylor resolution} coincides with the Taylor resolution whenever each $I_i$ is principal. More recently, VandeBogert \cite{Keller} proved that if the input resolutions admit DG algebra structures, then the generalized Taylor resolution also inherits a DG algebra structure. VandeBogert's product formula on the generalized Taylor resolution was missing a sign, an imprecision that we fix in this paper. 

The existence of a  DG $\Gamma$-structure on the generalized Taylor resolution has not yet been addressed in the literature. The purpose of this article is to fill this gap. Specifically, we prove that the generalized Taylor resolution $F_1 * \cdots * F_r$, formed from DG $\Gamma$-algebra resolutions $F_i$ of $R/I_i$, admits a natural structure of a DG $\Gamma$-algebra. Our construction of the divided powers structure on the generalized Taylor resolution mirrors the construction of the divided powers structure on the tensor product of DG $\Gamma$-algebras. 

The paper is organized as follows. In Section 2 we recall the necessary background information such as the definition of DG $\Gamma$-algebra and the construction of the Taylor resolution and its generalized version. We then give a correct proof of the existence of a DG algebra structure on the generalized Taylor resolution.

Section 3 contains our main result, i.e. the generalized Taylor resolution admits a DG $\Gamma$-algebra structure whenever the input resolutions do. Our proof relies on a general descent lemma for divided powers under DG algebra maps that become isomorphisms after localization. This allows us to transfer the DG $\Gamma$-structure from the tensor products of the input resolutions to the generalized Taylor resolution.

In Section 4, we verify that this construction recovers the canonical DG $\Gamma$-structure on the Taylor resolution in the special case where each $I_i$ is principal and $F_i$ is the Koszul resolution of $R/I_i$. 

In Section 5, we present two applications. First, we show that our construction can be used to compute the homotopy Lie algebra of rings of the form $R/(I_1+\cdots+I_r)$ where $I_1,\ldots, I_r$ are monomial ideals, generalizing a result of Avramov, see \cite[proof of Theorem 1]{Avramov}. Second, in the case of squarefree monomial ideals, we show that if the minimal free resolution of $R/(I_1+\cdots+I_r)$ is a DG $\Gamma$-algebra, then it is isomorphic to a DG $\Gamma$-quotient of the generalized Taylor resolution, provided the input resolutions are also DG $\Gamma$-algebras. This complements a result of VandeBogert \cite[Proposition 4.8]{Keller}, which was a generalization of a result of Katth\"{a}n \cite[Theorem 3.6]{katthan}. We conclude by showing that if the minimal free resolution of a squarefree ideal is a DG $\Gamma$-algebra, then the divided powers structure on the Scarf subcomplex is unique, complementing a result of Katth\"{a}n \cite[Proposition 3.1]{katthan}

\section{Background}

Let $R=\kk[x_1,\ldots,x_n]$ where $\kk$ is a field. We consider $R$ to be $\mathbb{Z}^n$-graded by giving the variable $x_i$ (multi)degree $(0,\ldots,0,1,0,\ldots,0)$, where the $1$ is in position $i$.

\begin{remark}
Every complex over $R$ will be considered to be bigraded. The homological degree of an element $x$ will be denoted by $|x|$, while the internal (multi)degree will be denoted by $\deg x$. For convenience we denote the internal degree of $x$ by the unique monomial with the same $\mathbb{Z}^n$-degree.
\end{remark}

\subsection{DG Algebras} A DG algebra, or differential graded algebra, over $R$ is a complex $A_\bullet$ of finitely generated free $R$-modules with differential $\partial$ and with a unitary, associative multiplication that respects both the homological grading and the internal grading and such that
\begin{enumerate}
\item $\partial(ab)=\partial(a)b+(-1)^{|a|}a\partial(b)$, for homogeneous elements $a,b\in A$,
\item $ab=(-1)^{|a||b|}ba$, for homogeneous elements $a,b\in A$,
\item $a^2=0$, if $a$ is a homogeneous element of odd homological degree.
\end{enumerate}

\subsection{DG $\Gamma$-Algebras}\label{DGGammaDef} We recall the definition of DG $\Gamma$-algebra, or DG algebra with divided powers. A DG algebra $A$ over $R$ is said to be a DG $\Gamma$-algebra if to every bihomogeneous element $x\in A$ of even positive homological degree, there is an associated sequence of elements $x^{(k)}\in A$ for $k=0,1,2,\ldots$ satisfying
\begin{enumerate}
\item $x^{(0)}=1, x^{(1)}=x, |x^{(k)}|=k|x|, \deg x^{(k)}=(\deg x)^k$,
\item $x^{(h)}x^{(k)}=\binom{h+k}{h}x^{(h+k)}$,
\item if $x$ and $y$ have the same bidegree, then
\[
(x+y)^{(k)}=\sum_{i+j=k}x^{(i)}y^{(j)},
\]
\item for $k\geq2$ and $y$ bihomogeneous
\[
(xy)^{(k)}=\begin{cases}
0\quad\text{if $|x|$ and $|y|$ are odd}\\
x^ky^{(k)}\quad\text{if $|x|$ is even and $|y|$ is even and positive},
\end{cases}
\]
\item 
\[
(x^{(h)})^{(k)}=\frac{(hk)!}{k!(h!)^k}x^{(hk)},
\]
\item for $k\geq 1$,
\[
\partial(x^{(k)})=x^{(k-1)}\partial(x),
\]
where $\partial$ is the differential of $A$.
\end{enumerate}

\subsection{The Taylor Resolution} Let $u_1,\ldots,u_r$ be monomials in $R$. Let $T_i=R^{\binom{r}{i}}$ for $i=0, \ldots, r$. A basis of $R^{\binom{r}{i}}$ is given by $\{e_F\mid |F|=i,F\subseteq[r]\}$.
We define maps $\partial_i:T_i\rightarrow T_{i-1}$ for $i=1,\ldots,r$ by 
\[
\partial_i (e_F) = \sum_{i \in F}(-1)^{\alpha(F,i)} \frac{m_F}{m_{F \backslash\{i\}}} e_{F \backslash\{i\}} , 
\]
where $\alpha(F,i)= |\{j \in F : j< i\}|$ and $m_F$ is defined to be the least common multiple of the monomials $u_i$ with $i\in F$. The multidegree of $e_F$ is defined to be $m_F$. Taylor proved in \cite{Taylor} that $(T_\bullet,\partial_\bullet)$ is a free resolution of $R/(u_1,\ldots,u_r)$.

Gemeda proved in \cite{Gemeda} that the Taylor resolution admits a DG algebra structure with the following product
\[ e_{V}e_{W}=
\begin{cases}
(-1)^{\alpha(V,W)}\frac{m_{V}m_{W}}{m_{V \cup W}}e_{V \cup W} & V\cap W = \emptyset\\
0 & V\cap W \neq \emptyset\\
\end{cases}
\]
where $\alpha(V, W)= |\{(i,j) \in V \times W \mid j<i\}|$.

In \cite[Lemma 6]{Avramov}, Avramov proved that the Taylor resolution admits a unique structure of a DG $\Gamma$-algebra where $e_F^{(r)}=0$ for $r\geq2$. The divided powers of a linear combination of basis elements is computed using the axioms of DG $\Gamma$-algebra.

\subsection{The Generalized Taylor Resolution} We now recall the construction of the generalized Taylor resolution. Let $I_1,\ldots, I_r$ be monomial ideals of $R$. Let $F_i$ be a multigraded resolution of $R/I_i$. If $f_i$ is a multigraded element of $F_i$, we denote its multidegree (as a monomial of $R$) as $m_{f_i}$. The differential of $f_i$ can be written as follows
\begin{equation}\label{eq:diffGenTayComps}
\partial^{F_i}(f_i)=\sum_{j=1}^{n_i}\alpha_{i,j}\frac{m_{f_i}}{m_{a_{i,j}}}a_{i,j},
\end{equation}
where $\alpha_{i,j}\in\kk$ and $a_{i,j}$ for $j=1,\ldots,n_i$ is a free basis of $F_i$.

The Generalized Taylor resolution, denoted by $F_1*\cdots* F_r$, is defined as follows. As bigraded $R$-modules $F_1*\cdots *F_r=F_1\otimes_R\cdots\otimes_R F_r$, and we denote a simple tensor $f_1\otimes\cdots\otimes f_r$ as $f_1*\cdots*f_r$ for $f_i\in F_i$. The multidegree of $f_1*\cdots*f_r$ is defined to be $[m_{f_1},\ldots,m_{f_r}]$, where the brackets denote the least common multiple. The differential, on bigraded elements, is defined as follows
\begin{align*}
&\partial^{F_1*\cdots*F_r}(f_1*\cdots*f_r)=\\
&\sum_{i=1}^r(-1)^{|f_1|+\cdots+|f_{i-1}|}\sum_{j=1}^{n_i}\alpha_{i,j}\frac{[m_{f_1},\ldots,m_{f_r}]}{[m_{f_1},\ldots,m_{f_{i-1}},m_{a_{i,j}},m_{f_{i+1}},\ldots,m_{f_r}]}f_1*\cdots*f_{i-1}*a_{i,j}*f_{i+1}*\cdots*f_r,
\end{align*}
where the coefficients $\alpha_{i,j}$ are defined in \Cref{eq:diffGenTayComps}. Herzog proved in \cite[Theorem 2.1]{herzog2007} that $F_1*\cdots *F_r$ is a multigraded free resolution of $R/(I_1+\cdots+I_r)$.

VandeBogert proved in \cite[Theorem 4.5]{Keller} that whenever the $F_i$'s are DG algebras, the resolution $F_1*\cdots*F_r$ also admits a DG algebra structure. The formula that he provides in \cite[Theorem 4.5]{Keller} is missing a sign. Below we state the correct formula and provide a correct proof.

\begin{theorem}\label{thm:DGA}
Let $F_1,\ldots, F_r$ be DG algebra resolutions of $R/I_1,\ldots, R/I_r$. Let $f,f'\in F_i$ be bigraded elements and consider the product $ff'$
\begin{equation}\label{eq:genTayDGA}
ff'=\sum_{j_i=1} ^{n_i}\alpha_{i,j_i,f,f'}\frac{m_fm_{f'}}{m_{a_{i,j_i}}}a_{i,j_i},
\end{equation}
with $\alpha_{i,j_i,f,f'}\in \kk$. Let $f_1\in F_1,\ldots, f_r\in F_r$ be bigraded elements, the following product defines a DG algebra structure on $F_1*\cdots *F_r$
\[
(f_1*\cdots*f_r)(f_1'*\cdots*f_r')=
\]
\[
(-1)^{\displaystyle\sum_{i=1}^r|f_i'|\sum_{j=i+1}^r|f_j|}\sum_{\substack{i=1,\ldots,r\\j_i=1,\ldots,n_i}}\alpha_{1,j_1,f_1,f_1'}\cdots\alpha_{r,j_r,f_r,f_r'}\frac{[m_{f_1},\ldots,m_{f_r}][m_{f_1'},\ldots,m_{f_r'}]}{[m_{a_{1,j_1}},\ldots,m_{a_{r,j_r}}]}a_{1,j_1}*\cdots*a_{r,j_r}
\]
\end{theorem}

\begin{proof}
For simplicity we show the proof for $r=2$. Consider the map $\phi:F_1\otimes_RF_2\rightarrow F_1*F_2$, defined on bigraded elements as $f_1\otimes f_2\mapsto (m_{f_1},m_{f_2})f_1*f_2$, where the parentheses denote the greatest common divisor. This map is a chain map by \cite[(3)]{herzog2007}. Let $S$ be the set of monomials of $R$ (including 1). It is clear that $S^{-1}\phi$ is an isomorphism of complexes. Therefore by \cite[Lemma 4.2]{Keller}, to prove the theorem it suffices to show that $\phi((f_1\otimes f_2)(f_1'\otimes f_2'))=\phi(f_1\otimes f_2)\phi(f_1'\otimes f_2')$, where the product on the left hand side of the equality is given by the standard DG algebra structure of the tensor product, while the product in the right hand side of the equality is the one defined in the statement of the theorem. We check the equality below $\phi(f_1\otimes f_2)\phi(f_1'\otimes f_2')$
\begin{align*}
&=(m_{f_1},m_{f_2})(m_{f_1'},m_{f_2'})(f_1*f_2)(f_1'*f_2')\\
&=(-1)^{|f_1'||f_2|}(m_{f_1},m_{f_2})(m_{f_1'},m_{f_2'})\sum_{\substack{j_1=1,\ldots,n_1\\j_2=1,\ldots,n_2}}\alpha_{1,j_1,f_1,f_1'}\alpha_{2,j_2,f_2,f_2'}\frac{[m_{f_1},m_{f_2}][m_{f_1'},m_{f_2'}]}{[m_{a_{1,j_1}},m_{a_{2,j_2}}]}a_{1,j_1}*a_{2,j_2}\\
&=(-1)^{|f_1'||f_2|}\sum_{\substack{j_1=1,\ldots,n_1\\j_2=1,\ldots,n_2}}\alpha_{1,j_1,f_1,f_1'}\alpha_{2,j_2,f_2,f_2'}\frac{m_{f_1}m_{f_1'}}{m_{a_{1,j_1}}}\frac{m_{f_2}m_{f_2'}}{m_{a_{2,j_2}}}(m_{a_{1,j_1}},m_{a_{2,j_2}})a_{1,j_1}*a_{2,j_2}\\
&=(-1)^{|f_1'||f_2|}\sum_{\substack{j_1=1,\ldots,n_1\\j_2=1,\ldots,n_2}}\alpha_{1,j_1,f_1,f_1'}\alpha_{2,j_2,f_2,f_2'}\frac{m_{f_1}m_{f_1'}}{m_{a_{1,j_1}}}\frac{m_{f_2}m_{f_2'}}{m_{a_{2,j_2}}}\phi(a_{1,j_1}\otimes a_{2,j_2})\\
&=(-1)^{|f_1'||f_2|}\phi((f_1f_1')\otimes(f_2f_2'))\\
&=(-1)^{|f_1'||f_2|}(-1)^{|f_1'||f_2|}\phi((f_1\otimes f_2)(f_1'\otimes f_2)\\
&=\phi((f_1\otimes f_2)(f_1'\otimes f_2)),
\end{align*}
where the 5th equality follows from \Cref{eq:genTayDGA}.
\end{proof}

\section{Main Result}

In this section we prove our main result, namely that the generalized Taylor resolution $F_1*\cdots*F_r$ admits a structure of a DG $\Gamma$-algebra whenever the components $F_1,\ldots, F_r$ do. The next Lemma will be the key ingredient in the proof of our main result.

\begin{lemma}\label{lem:transfer}
Let $F,G$ be complexes of free $R$-modules with a DG algebra structure such that $F$ also admits a compatible divided powers structure.  Let $\phi:F\rightarrow G$ be a DG algebra map. For each $y\in G$ of even positive degree fix a sequence $y^{(i)}$ for all $i\geq0$. Let $S$ be a multiplicative closed subset of $R$. If $S^{-1}\phi$ is an isomorphism of DG algebras and
\[
\phi(x^{(i)})=\phi(x)^{(i)}
\]
for all $x\in F$ of even positive degree and for all $i\geq0$, then $G$ is a DG $\Gamma$-algebra with divided powers given by $y^{(i)}$ for all $y\in G$ of even positive degree and for all $i\geq0$.
\end{lemma}

\begin{proof}
We first prove the following equality
\begin{equation}\label{eq:claim}
\left(S^{-1}\phi\right)^{-1}\left(\frac{y^{(n)}}{1}\right) = \left(\left(S^{-1}\phi\right)^{-1}\left(\frac{y}{1}\right)\right)^{(n)},
\end{equation}
for all $y\in G$ of even positive degree. Equation \Cref{eq:claim} is equivalent to
\begin{equation*}
  \frac{y^{(n)}}{1} = S^{-1}\phi \left( (S^{-1}\phi)^{-1} \left( \frac{y}{1} \right) \right)^{(n)}.
\end{equation*}
Set
\begin{equation*}
  \frac{a}{b} = (S^{-1}\phi)^{-1} \left( \frac{y}{1} \right),
\end{equation*}
so that:
\begin{equation*}
  \frac{y}{1} = S^{-1}\phi \left( \frac{a}{b} \right) = \frac{\phi(a)}{b}.
\end{equation*}
Thus there exists $s \in S$ such that
$s(yb - \phi(a)) = 0$.
Since $F,G$ are complexes of free $R$-modules over an integral domain, they are torsion free; thus,
$yb = \phi(a)$.
Moreover, by the definition of divided powers in a localization one has
\[
\left( (S^{-1}\phi)^{-1}\left(\frac{y}{1}\right) \right)^{(n)} = \left(\frac{a}{b}\right)^{(n)} = \frac{a^{(n)}}{b^n}.
\]
Therefore
\[
S^{-1} \phi \left( \frac{a^{(n)}}{b^n} \right) =\frac{\phi(a^{(n)})}{b^n} = \frac{\phi(a)^{(n)}}{b^n} = \frac{(yb)^{(n)}}{b^n} = \frac{b^n y^{(n)}}{b^n} = \frac{y^{(n)}}{1},
\]
where the first equality follows from the definition of $S^{-1}\phi$, the second from the hypothesis, the third was follows from the discussion above, the fourth one from the divided powers algebra axioms, and the last one is obvious. This concludes the proof that \Cref{eq:claim} holds.

Now we are going to prove that $G$ satisfies the DG $\Gamma$-algebra axioms from \ref{DGGammaDef}. We will only show Axiom 2 and 6, the remaining ones being similar.

\textbf{Axiom 2:} Let $y \in G$ such that $|y|$ is even and positive. We need to prove that
\[
y^{(n)} y^{(m)} - \binom{n+m}{n} y^{(n+m)}=0.
\]
By the properties of localization and since $F$ and $G$ are complexes of torsion-free modules, this is equivalent to:
\[
\frac{y^{(n)}}{1} \cdot \frac{y^{(m)}}{1} - \binom{n+m}{n} \frac{y^{(n+m)}}{1} = 0.
\]
Since $\left(S^{-1}\phi\right)^{-1}$ is an isomorphism, the previous display is equivalent to
\[
\left(S^{-1}\phi\right)^{-1}\left(\frac{y^{(n)}}{1}\right) \cdot \left(S^{-1}\phi\right)^{-1}\left(\frac{y^{(m)}}{1}\right) - \binom{n+m}{n} \left(S^{-1}\phi\right)^{-1}\left(\frac{y^{(n+m)}}{1}\right) = 0.
\]
Thus from equation \eqref{eq:claim} the previous display is equivalent to
\[
\left(\left(S^{-1}\phi\right)^{-1}\left(\frac{y}{1}\right)\right)^{(n)} \left(\left(S^{-1}\phi\right)^{-1}\left(\frac{y}{1}\right)\right)^{(m)} - \binom{n+m}{n} \left(\left(S^{-1}\phi\right)^{-1}\left(\frac{y}{1}\right)\right)^{(n+m)} = 0.
\]
Since $S^{-1}F$ is a DG $\Gamma$-algebra, Axiom 2 holds in $S^{-1}F$, therefore the equality in the previous display follows.

\textbf{Axiom 6:} We have to verify that for every $y \in G$ of even positive degree, $\partial^{G}(y^{(n)}) = \partial^{G}(y)y^{(n-1)}$. The previous equality is equivalent to 
\[
\frac{\partial^G(y^{(n)})}{1} - \frac{\partial^G(y)}{1} \cdot \frac{y^{(n-1)}}{1} = 0
\]
which in turn is equivalent to
\begin{equation}\label{eq:DiffComp}
\partial^{S^{-1}G} \left(\frac{y^{(n)}}{1}\right) - \partial^{S^{-1}G}\left(\frac{y}{1}\right) \cdot \frac{y^{(n-1)}}{1} = 0.
\end{equation}
Since $\partial^G\phi = \phi\partial^F$ 
it follows that:
\[
\partial^{S^{-1}G} = S^{-1}\phi \circ \partial^{S^{-1}F} \circ \left(S^{-1}\phi\right)^{-1}.
\]
Therefore, we can rewrite \eqref{eq:DiffComp} as
\[
S^{-1}\phi \partial^{S^{-1}F} \left(S^{-1}\phi\right)^{-1}\left(\frac{y^{(n)}}{1}\right) - S^{-1}\phi \partial^{S^{-1}F} \left(S^{-1}\phi\right)^{-1}\left(\frac{y}{1}\right) \cdot\left(\frac{y^{(n-1)}}{1}\right) = 0
\]
Multiplying by $\left(S^{-1}\phi\right)^{-1}$ on the left, we obtain:
\[
\partial^{S^{-1}F}\left(S^{-1}\phi\right)^{-1}\left(\frac{y^{(n)}}{1}\right) - \partial^{S^{-1}F}\left(S^{-1}\phi\right)^{-1}\left(\frac{y}{1}\right)\left(S^{-1}\phi\right)^{-1}\left(\frac{y^{(n-1)}}{1}\right) = 0.
\]
Applying \eqref{eq:claim} yields
\[
\partial^{S^{-1}F}\left(\left(\left(S^{-1}\phi\right)^{-1}\left(\frac{y}{1}\right)\right)^{(n)}\right) - \partial^{S^{-1}F}\left(\left(S^{-1}\phi\right)^{-1}\left(\frac{y}{1}\right)\right)\left(\left(S^{-1}\phi\right)^{-1}\left(\frac{y}{1}\right)\right)^{(n-1)} = 0.
\]
The previous equality holds since $S^{-1}F$ is a DG $\Gamma$-algebra.\qedhere

\end{proof}

\begin{remark}
Let $F_1,\ldots,F_r$ be DG $\Gamma$-algebras. Then $F_1\otimes_R\cdots\otimes_R F_r$ is a DG $\Gamma$-algebra, where the divided power structure is given by
\[
(f_1\otimes\cdots\otimes f_r)^{(i)}=\begin{cases}
0,\quad\text{if there is an $k$ such that $|f_k|$ is odd}\\
f_1^i\otimes\cdots f_{j-1}^i\otimes f_j^{(i)}\otimes f_{j+1}^i\otimes\cdots\otimes f_r^i,\quad\text{$|f_k|$ even $\forall k$ and $j$ smallest such that $|f_j|>0$.}
\end{cases}
\]
This is the unique DG $\Gamma$-algebra structure for which the canonical maps $F_i\rightarrow F_1\otimes_R\cdots\otimes_R F_r$ are DG $\Gamma$-algebra maps for all $i$'s. See \cite[Proposition 1.7.10]{GulliksenLevin}
\end{remark}

\begin{remark}
We point out that it follows from the DG $\Gamma$-algebra axioms that if $f_1,\ldots,f_r$ are even and $f_l,f_j$ have positive degree, then
\[
f_1^i\otimes\cdots f_{j-1}^i\otimes f_j^{(i)}\otimes f_{j+1}^i\otimes\cdots\otimes f_r^i=(f_1\otimes\cdots\otimes f_r)^{(i)}=f_1^i\otimes\cdots f_{l-1}^i\otimes f_l^{(i)}\otimes f_{l+1}^i\otimes\cdots\otimes f_r^i.
\]
The same will be true for the DG $\Gamma$-algebra structure on the generalized Taylor resolution given below. Defining it first by taking the smallest $j$ such $|f_j|>0$ allows us to have a well-posed definition, but it will follow from the DG $\Gamma$-algebra axioms that taking $j$ to be the smallest is not needed.
\end{remark}

\begin{theorem}
Let $I_1,\ldots, I_r$ be monomial ideals. Let $F_i$ be a free DG $\Gamma$-algebra resolution of $R/I_i$ for all $i=1,\ldots,r$. Then, the resolution $F_1*\cdots *F_r$ is a DG $\Gamma$-algebra where the divided power structure is given, on elements of positive and even homological degree, by
\[
(f_1*\cdots*f_r)^{(i)}=\begin{cases}
0\quad\text{if there is an $k$ such that $|f_k|$ is odd}\\
f_1^i*\cdots f_{j-1}^i*f_j^{(i)}*f_{j+1}^i*\cdots*f_r^i\quad\text{$|f_k|$ even $\forall k$ and $j$ smallest such that $|f_j|>0$.}
\end{cases}
\]
For the divided power of a linear combination of the elements above one uses the axioms defining divided power algebras.

Moreover this is the unique DG $\Gamma$-structure for which the inclusion maps $F_i\rightarrow F_1*\cdots*F_r$ are maps of DG $\Gamma$-algebras for all $i=1,\ldots,r$.
\end{theorem}

\begin{proof}
The map
\begin{align*}
\varphi:&F_1\otimes_R\cdots\otimes_R F_r\rightarrow F_1*\cdots*F_r\\
&f_1\otimes\cdots\otimes f_r\longmapsto(m_{f_1},\ldots,m_{f_r})f_1*\cdots*f_r
\end{align*}
is a map of DG algebras by \cite[Theorem 4.5]{Keller}.
It is clearly invertible when localized at the multiplicative subset consisting of all monomials, therefore by \Cref{lem:transfer} it suffices to prove
\[
\varphi((f_1\otimes\cdots\otimes f_r)^{(i)})=\varphi(f_1\otimes\cdots\otimes f_r)^{(i)}.
\]

Assuming that $j$ is the smallest index such that $|f_j|>0$, the following chain of equalities gives the desired result
\begin{align*}
\varphi((f_1\otimes\cdots\otimes f_r)^{(i)})&=\varphi(f_1^i\otimes\cdots\otimes f_j^{(i)}\otimes\cdots\otimes f_r^i)\\
&=(m_{f_1^i},\ldots,m_{f_j^{(i)}},\ldots,m_{f_r^i})f_1^i*\cdots*f_j^{(i)}*\cdots*f_r^i\\
&=(m_{f_1}^i,\ldots,m_{f_j}^i,\ldots,m_{f_r}^i)f_1^i*\cdots*f_j^{(i)}*\cdots*f_r^i\\
&=(m_{f_1},\ldots,m_{f_j},\ldots,m_{f_r})^if_1^i*\cdots*f_j^{(i)}*\cdots*f_r^i\\
&=((m_{f_1},\ldots,m_{f_r})f_1*\cdots*f_r)^{(i)}\\
&=\varphi(f_1\otimes\cdots\otimes f_r)^{(i)}.
\end{align*}

We show that the DG $\Gamma$-algebra structure is unique. Assume that the generalize Taylor resolution admits a DG $\Gamma$-structure and assume that the canonical maps
\begin{align*}
\iota_i:&F_i\rightarrow F_1*\cdots*F_r\\
&f_i\mapsto1*\cdots1*f_i*1\cdots1
\end{align*}
are DG $\Gamma$-algebra maps for all $i$. Let $f_1*\cdots*f_r\in F_1*\cdots*F_r$ of even and positive homological degree and assume that $j$ is the smallest index such that $|f_j|>0$. Then
\begin{align*}
(f_1*\cdots *f_r)^{(i)}&=\left(\frac{[m_{f_1},\ldots, m_{f_r}]}{m_{f_1}\cdots m_{f_r}}\iota_1(f_1)\cdots\iota_r(f_r)\right)^{(i)}\\
&=\left(\frac{1}{(m_{f_1},\ldots,m_{f_r})}\iota_1(f_1)\cdots\iota_r(f_r)\right)^{(i)}\\
&=\frac{1}{(m_{f_1},\ldots,m_{f_r})^i}(\iota_1(f_1)\cdots\iota_r(f_r))^{(i)}\\
&=\frac{1}{(m_{f_1},\ldots,m_{f_r})^i}\iota_1(f_1)^i\cdots\iota_j(f_j)^{(i)}\cdots\iota_r(f_r)^i\\
&=\frac{1}{(m_{f_1},\ldots,m_{f_r})^i}\iota_1(f_1^i)\cdots\iota_j(f_j^{(i)})\cdots\iota_r(f_r^i)\\
&=\frac{1}{(m_{f_1},\ldots,m_{f_r})^i}\frac{m_{f_1^i}\cdots m_{f_j^{(i)}}\cdots m_{f_r^i}}{[m_{f_1^i},\ldots,m_{f_j^{(i)}},\ldots,m_{f_r^i}]}f_1^i*\cdots*f_j^{(i)}*\cdots*f_r^i\\
&=\frac{1}{(m_{f_1},\ldots,m_{f_r})^i}\frac{m_{f_1}^i\cdots m_{f_j}^i\cdots m_{f_r}^i}{[m_{f_1}^i,\ldots,m_{f_j}^i,\ldots,m_{f_r}^i]}f_1^i*\cdots*f_j^{(i)}*\cdots*f_r^i\\
&=\frac{1}{(m_{f_1},\ldots,m_{f_r})^i}(m_{f_1}^i,\ldots,m_{f_r}^i)f_1^i*\cdots*f_j^{(i)}*\cdots*f_r^i\\
&=f_1^i*\cdots*f_j^{(i)}*\cdots*f_r^i,
\end{align*}
where the first equality follows from \Cref{thm:DGA}.
\end{proof}

\section{Comparison to the Taylor Resolution}

In this section we show that if $I_1,\ldots,I_r$ are principal monomial ideals and $F_i$ is the Koszul resolution of $R/I_i$, then the Taylor resolution of $I_1+\cdots+I_r$ and the generalized Taylor resolution $F_1*\cdots*F_r$ are isomorphic as DG $\Gamma$-algebras. We point out that it was already observed by Herzog in \cite{herzog2007} that these complexes were isomorphic, even though he did not include a proof. VandeBogert also observed in \cite{Keller} that these complexes were isomorphic as DG algebras, but he also did not include a proof. Since VandeBogert's product formula was missing a sign, we have decided to show the proof of this claim in full detail.

\begin{theorem}
Let \( I = (u_1, \dots, u_r) \) be a monomial ideal in \( R = \kk[x_1, \dots, x_n] \). Let $I_i=(u_i)$ and $F_i$ be the Koszul resolution of $R/I_i$. Then, the Taylor resolution \( T_{\bullet} \) of $I$ and the generalized Taylor resolution 
\[
F_\bullet= F_1 * \cdots * F_r
\]
are isomorphic as DG \(\Gamma\)-algebras.
\end{theorem}

\begin{proof}
Let $F_i=(0\rightarrow Rf_i\rightarrow R\rightarrow0)$ be the Koszul resolution of $R/(u_i)$.
We construct a homomorphism
\[
\Phi: T_{\bullet} \to F_1 * \cdots * F_r
\]
and prove that it is an isomorphism of DG \( \Gamma \)-algebras. For a subset $\sigma \subseteq \{1, \dots, r\}$, we introduce the notation:
\[
f^*_{\sigma} := g_1*\cdots *g_r,
\]
where 
\[
g_i:=\begin{cases}f_i,\quad\mathrm{if\;}i\in\sigma\\
1,\quad\mathrm{if\;}i\not\in\sigma.
\end{cases}
\]
We define \( \Phi \) on basis elements of \( T_{\bullet} \) by $\Phi(e_\sigma) = f^*_{\sigma}$.
To verify that \( \Phi \) commutes with differentials, we need to show 
\begin{equation}\label{eq:cxMap}
\Phi(\partial^{T_{\bullet}}(e_\sigma)) = \partial^{F_\bullet}(\Phi(e_\sigma)).
\end{equation}
Computing the left side of \Cref{eq:cxMap} yields
\begin{align*}
\Phi(\partial^{T_{\bullet}}(e_\sigma)) &= \sum_{i \in \sigma} (-1)^{\alpha(i,\sigma)} \frac{m_\sigma}{m_{\sigma\setminus\{i\}}} \Phi(e_{\sigma \setminus \{i\}})\\
&=\sum_{i \in \sigma} (-1)^{\alpha(i,\sigma)} \frac{m_\sigma}{m_{\sigma\setminus\{i\}}} f^*_{\sigma\setminus\{i\}}.
\end{align*}
Next we compute the right side of \Cref{eq:cxMap}. First we notice the following
\[
\partial^{F_i}(g_i)=\begin{cases}
u_i,\quad i\in\sigma\\
0,\quad i\not\in\sigma.
\end{cases}
\]
Since $\partial^{F_\bullet}(\Phi(e_\sigma))=\partial^{F_\bullet}(f_\sigma^*)$, using the formula for the differential of the generalized Taylor resolution and the observation above one gets
\[
\partial^{F_\b}(f_\sigma^*)=\sum_{i\in\sigma}(-1)^{|g_1|+\cdots+|g_{i-1}|}\frac{[m_{g_1},\ldots,m_{g_r}]}{[m_{g_1},\ldots,m_{g_{i-1}},m_{g_{i+1}},\ldots,m_{g_r}]}g_1*\cdots g_{i-1}*1*g_{i+1}*\cdots g_r.
\]
Now we make the following elementary observations:
\[
g_1*\cdots g_{i-1}*1*g_{i+1}*\cdots g_r=f_{\sigma\setminus\{i\}}^*,
\]
and
\[
[m_{g_1},\ldots,m_{g_r}]=m_{\sigma},\quad [m_{g_1},\ldots,m_{g_{i-1}},m_{g_{i+1}},\ldots,m_{g_r}]=m_{\sigma\setminus\{i\}}.
\]
To conclude the proof of \Cref{eq:cxMap} it just remains to check the signs
\begin{align*}
|g_1|+\cdots+|g_{i-1}|&=\sum_{\substack{j\in\sigma\\j\leq i-1}}|f_j|\\
&=|\{j\in\sigma\mid j\leq i-1\}|\\
&=\alpha(i,\sigma),
\end{align*}
where the second equality follows since $|f_j|=1$ for all $j$'s. This concludes the proof that $\Phi$ commutes with the differential.

Now we check that $\Phi$ is a DG algebra map, i.e. we need to check that $\Phi(e_Pe_Q)=\Phi(e_P)\Phi(e_Q)$, that is we need to check that $\Phi(e_Pe_Q)=f_P^*f_Q^*$. We divide the proof into two cases.

\textbf{Case 1: $P\cap Q=\emptyset$}. In this case
\[
\Phi(e_Pe_Q)=(-1)^{\alpha(P,Q)}(m_P,m_Q)\Phi(e_{P\cup Q})=(-1)^{\alpha(P,Q)}(m_P,m_Q) f_{P\cup Q}^*.
\]
Next we compute the product $f_P^*f_Q^*$. We introduce the notation $(f_P^*)_i$ for the $i$th component of $f_P^*$, which is 1 if $i\not\in P$ and $f_i$ otherwise. Similarly for $Q$. We begin by noticing that
\[
(f_P^*)_i(f_Q^*)_i=\begin{cases}
1\quad i\not\in P\cap Q\\
(f_P^*)_i\quad i\in P\\
(f_Q^*)_i\quad i\in Q.
\end{cases}
\]
Therefore, using the formula for the product of the generalized Taylor resolution one gets
\[
f_P^*f_Q^*=(-1)^{\displaystyle\sum_{i=1}^r|(f_Q^*)_i|\sum_{j=i+1}^r|(f_P^*)_j|}\frac{m_{f_P^*}m_{f_Q^*}}{m_{f_{P\cup Q}^*}}f_{P\cup Q}^*.
\]
We first observe that $m_{f_P^*}=m_P$, and therefore
\[
\frac{m_{f_P^*}m_{f_Q^*}}{m_{f_{P\cup Q}^*}}=\frac{m_{P}m_{Q}}{m_{P\cup Q}}=\frac{m_{P}m_{Q}}{[m_P,m_Q]}=(m_P,m_Q).
\]
Next we check the the signs coincide
\begin{align*}
\sum_{i=1}^r|(f_Q^*)_i|\sum_{j=i+1}^r|(f_P^*)_j|&=\sum_{i=1}^r|(f_Q^*)_i|\cdot|\{j\in P\mid i+1\leq j\leq r\}|\\
&=\sum_{i\in Q}|\{j\in P\mid i+1\leq j\leq r\}|\\
&=|\{(j,i)\in  P\times Q\mid i<j\}|\\
&=\alpha(P,Q).
\end{align*}

\textbf{Case 2: $P\cap Q\neq\emptyset$.} In this case $\Phi(e_Pe_Q)=\Phi(0)=0$. Moreover, if $i\in P\cap Q$, then $(f_P^*)_i(f_Q^*)_i=0$, and therefore $f_P^*f_Q^*=0$ as well. This concludes the proof that $\Phi$ is a map of DG algebras.

Now, to show that \( \Phi \) is a map of DG \(\Gamma\)-algebras, it suffices to show that $\Phi(e_\sigma^{(i)})=\Phi(e_\sigma)^{(i)}$ for all $i$'s.  If \( i = 0,1 \), then the desired equality follows immediately from the first axiom defining $\Gamma$-algebras.
So we will assume that \( i \geq 2 \). Avramov’s proof of \cite[Lemma 6]{Avramov} shows that divided powers vanish
\[
e_\sigma^{(i)} = 0\quad\text{for all }i\geq2.
\]
Therefore $\Phi(e_\sigma^{(i)}) = \Phi(0) = 0$. Thus, it remains to show
\[
\Phi(e_\sigma)^{(i)}=(f_\sigma^*)^{(i)}=0
\]
Recall that $f_\sigma^*=g_1*\cdots*g_r$, and assume that $j$ is the smallest index such that $|g_j|>0$. Then
\[
(f_\sigma^*)^{(i)}=(g_1*\cdots*g_r)^{(i)}=g_1^i*\cdots* g_j^{(i)}*\cdots g_r^i.
\]
Since $|g_j|>0$ it follows that $g_j=f_j$. Since the resolution $F_j$ is zero in degree $\geq2$, it follows that $g_j^{(i)}=0$, and therefore $(f_\sigma^*)^{(i)}=0$. This concludes the proof that $\Phi$ respects divided powers.
\end{proof}

\section{Some Applications}
In this section we include some applications of our main result. 
\subsection{Homotopy Lie Algebras} The first application concerns homotopy Lie algebras, see \cite[Chapter 10]{IFR} for the definition and basic properties. If $I$ is a monomial ideal of $R$, we denote the homotopy Lie algebra of $R/I$ by $\pi(R/I)$. If $D$ is a DG $\Gamma$-algebra such that $D_0$ is noetherian and the $D_0$-module $\mathrm{H}_n(D)$ is finite for every $n\in\mathbb{Z}$, we denote the homotopy Lie algebra of $D$ by $\pi(D)$, see \cite[Remark 7]{Avramov} for more details. In the proof of \cite[Theorem 1]{Avramov}, Avramov proves that if $I$ is a monomial ideal of $R$ minimally generated by monomials of degree at least 2, then $\pi^{\geq2}(R/I)\cong\pi(T_\bullet\otimes_R\kk)$, where $T_\bullet$ is the Taylor resolution of $I$. The next proposition generalizes Avramov's result.

\begin{proposition}
Let $I_1,\ldots,I_r$ be monomial ideals of $R$ minimally generated by monomials of degree at least 2. Let $F_i$ be a DG $\Gamma$-algebra resolution of $R/I_i$ for all $i$'s. Then, there is an isomorphism of graded Lie algebras
\[
\pi^{\geq2}(R/I)\cong\pi((F_1*\cdots*F_r)\otimes_R\kk).
\]
\end{proposition}
\begin{proof}
Let $K^R$ be the Koszul complex over the variables of $R$. Let $\varepsilon:K^R\rightarrow\kk$ and $\tau:F_1*\cdots*F_r\rightarrow R/I$ be the augmentation maps of $K^R$ and $F_1*\cdots*F_r$. These maps induce the quasi-isomorphisms below
\[
K^{R/I}= R/I\otimes_RK^R\xleftarrow{\;\;\tau\otimes K^R\;\;} (F_1*\cdots*F_r)\otimes_R K^R\xrightarrow{\;\; (F_1*\cdots*F_r)\otimes\varepsilon\;\;}(F_1*\cdots*F_r)\otimes_R\kk,
\]
where $K^{R/I}$ is the Koszul complex over the images of the variables of $R$ in $R/I$ and the equality follows since $I$ is generated by monomials of degree at least 2. The Koszul complex $K^{R/I}$ is a DG $\Gamma$-algebra and it was proved in \cite[Lemma 8]{Avramov} that $\pi^{\geq2}(R/I)\cong\pi(K^{R/I})$. Since quasi-isomorphic DG $\Gamma$-algebras have isomorphic homotopy Lie algebras the desired result now follows.
\end{proof}

\subsection{Squarefree monomial ideals}
In this subsection we complement a result of VandeBogert \cite[Proposition 4.8]{Keller}, which was a generalization of a result of Katth\"{a}n \cite[Theorem 3.6]{katthan}.

\begin{remark}We start by recalling the definition of the algebraic Scarf complex. The algebraic Scarf complex is the subcomplex of the Taylor resolution generated by all the basis elements having unique multidegree. It is known that the Scarf complex is always a subcomplex of the minimal free resolution of a monomial ideal. For more details about the Scarf complex see \cite[Section 6.2]{combCommAlg}.
\end{remark}

\begin{remark}\label{rmk:sqf}
If $F$ is a multigraded free $R$-module whose generators have squarefree multidegree, then for every multihomogeneous $f\in F$ there exists a monomial $u_f\in R$ and a multihomogeneous element of $F$ denoted by $|f|_{\sqf}$ such that $f=u_f|f|_\sqf$ and the multidegree of $|f|_\sqf$ is $(m_f,x_1\cdots x_n)$. The monomial $u_f$ and the element $|f|_\sqf$ are unique. See \cite[Lemma 1.4]{katthan}.

If $I$ is a squarefree monomial ideal of $R$ and $F_\b$ is a minimal DG algebra resolution of $R/I$, then every multihomogeneous element $f\in F$ of positive homological degree and squarefree multidegree can be written as
\[
f=\sum_{j=1}^l|g_{1,j}\cdots g_{|f|,j}|_\sqf
\]
where $g_{1,j},\ldots,g_{|f|,j}$ have homological degree 1 for all $j=1,\ldots,l$. See \cite[Proposition 3.4]{katthan} for a proof. This shows that a minimal DG algebra resolution of a squarefree monomial ideal is essentially generated in homological degree 1.
\end{remark}

\begin{theorem}
Let $I_1,\ldots,I_r$ be squarefree monomial ideals. Let $F_i$ be a multigraded DG $\Gamma$-algebra minimal free resolution of $R/I_i$ for $i=1,\ldots,r$. If the minimal free resolution $G_\b$ of $R/(I_1+\cdots +I_r)$ is a multigraded DG $\Gamma$-algebra, then there exists a DG $\Gamma$-ideal $J$ of $F_1*\cdots*F_r$ such that $G_\b\cong (F_1*\cdots *F_r)/J$ as multigraded DG $\Gamma$-algebras.
\end{theorem}

\begin{proof}
Set $F_\b=F_1*\cdots*F_r$. We assume the resolutions $F_1,\ldots,F_r,G_\b$ have a fixed basis. In the proof of \cite[Proposition 4.8]{Keller}, VandeBogert constructs a surjective DG algebra map $\psi:F_\b\rightarrow G_\b$. We recall this construction and show that $\psi$ preserves divided powers.

If $f_i\in F_i$, we denote by $f_i^*$ the element of $F_\b$ with $f_i$ in position $i$ and 1 everywhere else. We first notice that if $f_i\in F_i$ is of squarefree multidegree for all $i$, it follows from \Cref{thm:DGA} that
\[
f_1*\cdots*f_r=\frac{f_1^*\cdots f_r^*}{(m_{f_1},\ldots,m_{f_r})}=|f_1^*\cdots f_r^*|_\sqf.
\]
We now define $\psi$ in homological degree 1. If $g_i\in F_i$ is a basis element of homological degree 1, then we set $\psi(g_i^*)=\tilde{g}_i$ where $\tilde{g}_i$ is the element of the algebraic Scarf complex (as a subcomplex of $G_\b$) corresponding to $g_i$. We then extend by $R$-linearity.

If $f_i\in F_i$ is a basis element (and therefore of squarefree multidegree) of positive homological degree, then by \Cref{rmk:sqf}
\[
f_i=\sum_{j=1}^l|g_{1,j}\cdots g_{|f_i|,j}|_\sqf,
\]
where $g_{1,j},\ldots,g_{|f_i|,j}$ are elements of $F_i$ of homological degree 1 for all $j=1,\ldots,l$. We define 
\[
\psi(f_i^*)=\sum_{j=1}^l|\psi(g_{1,j}^*)\cdots \psi(g_{|f_i|,j}^*)|_\sqf,
\]
and extend by $R$-linearity.

If $f_1\in F_1,\ldots, f_r\in F_r$ are basis elements, then we define
\[
\psi(f_1*\cdots*f_r)=\frac{\psi(f_1^*)\cdots\psi(f_r^*)}{(m_{f_1},\ldots,m_{f_r})}=|\psi(f_1^*)\cdots\psi(f_r^*)|_\sqf,
\]
and extend by $R$-linearity.

%We now show that $\psi$ preserves divided powers. We first note that if $g_i\in F_i$ is an element of homological degree 1, then $g_i^{(m)}=0$ for $m\geq2$ since $0=g_i^m=m!g_i^{(m)}$. Similarly for elements of homological degree 1 of $G_\b$. Therefore it is obvious that $\psi((g_i^*)^{(m)})=\tilde{g_i}^{(m)}$ for all $m$. 

Now we show that $\psi$ preserves divided powers. Let $f_i\in F_i$ be a basis element of even and positive homological degree. Then by \Cref{rmk:sqf}
\[
f_i=\sum_{j=1}^l|g_{i,1,j}\cdots g_{i,|f_i|,j}|_\sqf
\]
for some elements of homological degree 1 $g_{i,1,j},\ldots,g_{i,|f_i|,j}$ for $j=1,\ldots,l$. By the axioms of DG $\Gamma$-algebras one has the following equality
\begin{equation}\label{eq:gamma}
f_i^{(m)}=\sum_{j_1+\cdots+j_l=m}|g_{i,1,j_1}\cdots g_{i,|f_i|,j_1}|_\sqf^{(j_1)}\cdots|g_{i,1,j_l}\cdots g_{i,|f_i|,j_l}|_\sqf^{(j_l)}.
\end{equation}
If $h\in F_i$ is of even positive homological degree then, since $h=u_h|h|_\sqf$, it follows that $h^{(m)}=0$ if and only if $|h|_\sqf^{(m)}=0$. We now note that by the axioms of divided powers algebra
\begin{equation}\label{eq:zeroPw}
(g_{i,1,j_k}\cdots g_{i,|f_i|,j_k})^{(j_k)}=0\quad\mathrm{for\;}j_k\geq2,\quad\mathrm{and\;}k=1,\ldots,l.
\end{equation}
Therefore \Cref{eq:gamma} can be rewritten as
\[
f_i^{(m)}=\sum_{\substack{\sigma\subseteq[m]\\|\sigma|=l}}\prod_{j\in\sigma}|g_{i,1,j}\cdots g_{i,|f_i|,j}|_\sqf.
\]
Therefore
\begin{align*}
\psi((f_i^*)^{(m)})&=\sum_{\substack{\sigma\subseteq[m]\\|\sigma|=l}}\prod_{j\in\sigma}\psi(|g_{i,1,j}^*\cdots g_{i,|f_i|,j}^*|_\sqf)\\
&=\sum_{\substack{\sigma\subseteq[m]\\|\sigma|=l}}\prod_{j\in\sigma}|\psi(g_{i,1,j}^*)\cdots \psi(g_{i,|f_i|,j}^*)|_\sqf.
\end{align*}
While
\begin{align*}
\psi(f_i^*)^{(m)}&=\left(\sum_{j=1}^l\psi(|g_{i,1,j}^*\cdots g_{i,|f_i|,j}^*|_\sqf)\right)^{(m)}\\
&=\left(\sum_{j=1}^l|\psi(g_{i,1,j}^*)\cdots \psi(g_{i,|f_i|,j}^*)|_\sqf\right)^{(m)}\\
&=\sum_{j_1+\cdots+j_l=m}|\psi(g_{i,1,j_1}^*)\cdots \psi(g_{i,|f_i|,j_1}^*)|_\sqf^{(j_1)}\cdots|\psi(g_{i,1,j_l}^*)\cdots \psi(g_{i,|f_i|,j_l}^*)|_\sqf^{(j_l)}\\
&=\sum_{\substack{\sigma\subseteq[m]\\|\sigma|=l}}\prod_{j\in\sigma}|\psi(g_{i,1,j}^*)\cdots \psi(g_{i,|f_i|,j}^*)|_\sqf,
\end{align*}
where the last equality follows from \Cref{eq:zeroPw}. This shows $\psi((f_i^*)^{(m)})=\psi(f_i^*)^{(m)}$ for all $m$.

Finally if $f_1\in F_1,\ldots, f_r\in F_r$ are basis elements of even degree, and $j$ is the smallest index such that $|f_j|>0$ and $k$ any positive integer, then
\begin{align*}
\psi((f_1*\cdots*f_r)^{(k)})&=\psi(f_1^k*\cdots*f_j^{(k)}*\cdots*f_r^k)\\
&=\frac{\psi((f_1^*)^k)\cdots\psi((f_j^*)^{(k)})\cdots\psi((f_r^*)^k)}{(m_{f_1^k},\ldots,m_{f_j^{(k)}},\ldots,m_{f_r^k})}\\
&=\frac{\psi(f_1^*)^k\cdots\psi(f_j^*)^{(k)}\cdots\psi(f_r^*)^k}{(m_{f_1},\ldots,m_{f_r})^k}\\
&=\left(\frac{\psi(f_1^*)\cdots\psi(f_j^*)\cdots\psi(f_r^*)}{(m_{f_1},\ldots,m_{f_r})}\right)^{(k)}\\
&=\psi(f_1*\cdots*f_r)^{(k)}.\qedhere
\end{align*}
\end{proof}

\begin{remark}
We point out that if the minimal free resolution of $R/I$ with $I$ squarefree has a DG algebra structure, then it has a divided powers algebra structure. Indeed, if $f$ is a basis element of the resolution of even positive homological degree, then
\[
f=\sum_{j=1}^l|g_{1,j}\cdots g_{|f|,j}|_\sqf,
\]
with $g_{1,j},\ldots,g_{|f|,j}$ of homological degree 1 for all $j$. The divided power $(g_{1,j}\cdots g_{|f|,j})^{(m)}$ needs to be defined as 0 for $m\geq2$ to satisfy the axioms of divided powers algebra. It follows that $|g_{1,j}\cdots g_{|f|,j}|_\sqf^{(m)}=0$ for $m\geq2$, and following the axioms one needs to set
\[
f^{(m)}=\sum_{\substack{\sigma\subseteq[m]\\|\sigma|=l}}\prod_{j\in\sigma}|g_{1,j}\cdots g_{|f_i|,j}|_\sqf.
\]
Then one can extend these divided powers to linear combinations of basis elements by using the axioms of DG $\Gamma$-algebra.
\end{remark}

In \cite[Proposition 3.1]{katthan}, Katth\"{a}n proved that if the minimal free resolution of $R/I$, where $I$ is a squarefree ideal, is a DG algebra, then the product induced on the Scarf subcomplex is unique and it matches the product on the Taylor resolution. We complement this result by showing that if this resolution admits the structure of a DG $\Gamma$-algebra, then the divided powers structure on the Scarf subcomplex matches the DG $\Gamma$-structure of the Taylor resolution. If $I$ is a monomial ideal generated by $r$ elements, then we fix basis elements of the Scarf complex $e_\sigma$ with $\sigma\subseteq[r]$ corresponding to the standard basis elements of the Taylor resolution.

\begin{proposition}
Let $I$ be a squarefree monomial ideal of $R$. Let $F_\b$ be a minimal DG $\Gamma$-algebra free resolution of $R/I$. If $e_\sigma$ is a basis element of the Scarf complex (as a subcomplex of $F_\b$) corresponding to a standard basis element of the Taylor resolution, then $e_\sigma^{(m)}=0$ for all $m\geq2$.
\end{proposition}

\begin{proof}
We first notice that the claim is true when $|\sigma|=1$. It follows from \cite[Proposition 3.1]{katthan} that the basis elements of the Scarf complex multiply exactly as they would in the Taylor resolution. Therefore
\[
\prod_{i\in\sigma}e_i=\pm d_\sigma e_\sigma
\]
for some monomial $d_\sigma$. Therefore
\[
0=\left(\prod_{i\in\sigma}e_i\right)^{(m)}=\pm d_{\sigma}^me_{\sigma}^{(m)},
\]
yielding the desired result.
\end{proof}

\bibliographystyle{amsplain}
\bibliography{biblio}

\end{document}